\documentclass{amsart}

\usepackage{amsmath, amssymb, amscd}

\newcommand{\FF}{\mathbb{F}}
\newcommand{\ZZ}{\mathbb{Z}}
\newcommand{\fa}{\mathfrak{a}}

\newcommand{\fd}{\mathfrak{d}}
\newcommand{\fs}{\mathfrak{s}}
\newcommand{\fj}{\mathfrak{j}}
\newcommand{\ml}{\mathcal{L}}
\newcommand{\mj}{\mathcal{J}}
\newcommand{\co}{\mathcal{O}}
\renewcommand{\d}{\partial}
\DeclareMathOperator{\id}{Id}
\DeclareMathOperator{\Hom}{Hom}
\DeclareMathOperator{\ad}{ad}
\DeclareMathOperator{\cur}{Cur}
\DeclareMathOperator{\Der}{Der}
\DeclareMathOperator{\inder}{Inder}
\DeclareMathOperator{\Tor}{Tor}
\DeclareMathOperator{\js}{JS}
\DeclareMathOperator{\jck}{JCK}
\DeclareMathOperator{\coef}{\mathcal{A}}
\DeclareMathOperator{\aut}{Aut}
\DeclareMathOperator{\sli}{sl}

\theoremstyle{plain}
\newtheorem{theorem}{Theorem}
\newtheorem{lemma}[theorem]{Lemma}
\newtheorem{proposition}[theorem]{Proposition}
\newtheorem{corollary}[theorem]{Corollary}

\theoremstyle{definition}

\newtheorem{example}[theorem]{Example}

\theoremstyle{remark}
\newtheorem{remark}[theorem]{Remark}

\numberwithin{theorem}{section}
\numberwithin{equation}{section}

\title{Simple Jordan conformal superalgebras}

\author{Victor~G.~Kac}
\address{Department of Mathematics, MIT, Cambridge MA 02139}
\email{kac@math.mit.edu}

\author{Alexander Retakh}
\address{Deparment of Mathematics, Stony Brook University, Stony Brook, NY 11794}
\email{retakh@math.sunysb.edu}

\begin{document}

\begin{abstract}
We classify simple finite Jordan conformal superalgebras and also establish preliminary results for the classification of simple finite Jordan pseudoalgebras.
\end{abstract}

\maketitle

\section*{Introduction}

Lie conformal superalgebras first appeared as algebraic structures that encode the singular part of the operator product expansion in two-dimensional conformal field theory.  Their generalization, the Lie pseudoalgebras, are closely connected to the Hamiltonian formalism in the theory of nonlinear evolutionary equations \cite{BDK}.

The conformal superalgebra (and more generally pseudoalgebra) formalism is also useful in dealing with other varieties of algebras.  In the present paper we classify simple Jordan conformal superalgebras and also establish preliminary results for the classification of simple Jordan pseudoalgebras.

It was shown in \cite{Zel} that simple finite Jordan conformal algebras must be current conformal algebras; this result was later extended to (non-super)pseudoalgebras in \cite{Kol}.  Both papers utilize the Tits--Kantor--Koecher (TKK) construction and rely on the previously known classification of Lie pseudoalgebras, with \cite{Kol} defining the TKK construction directly for pseudoalgebras.    

The world of Jordan conformal superalgebras is more complex.  Apart from the current superalgebras, we have one more series of superalgebras (related to Lie conformal superalgebras of contact type) and two exceptional superalgebras.
Here, to work directly in the category of pseudoalgebras, i.e. to use the pseudoalgebraic analog of the TKK construction, we would need to classify all possible short gradings of finite type simple pseudoalgebras, but the classification of the latter has not been done.  Instead, we take a shortcut and operate mostly on the level of annihilation algebras, whose classification can be reduced to that of simple linearly compact Jordan superalgebras, classified in \cite{CKj}.  As such we need only to relate the automorhisms and derivations of linerly compact Jordan superalgebras to those of their TKK envelopes.

We begin by estalishing preliminaries on Jordan algebras in Section~\ref{sec1}.  In Section~\ref{sec2}, we define pseudolalgebras and, specifically, conformal superalgebras and explicitly construct examples of simple finite Jordan conformal superalgebras.  Finally, in Section~\ref{sec3} we determine all possible annihilation algebras of simple finite Jordan pseudoalgebras and from this derive the classification of simple finite Jordan conformal superalgebras by showing that the list of examples, constructed in Section~\ref{sec2}, is complete.  The general case of Jordan pseudoalgebras will be treated elsewhere.

Throughout the paper the base field $\FF$ is algebraically closed of characteristic $0$.

\section{Preliminaries: Jordan superalgebras}\label{sec1}

\subsection{Superalgebras} Recall that a $\ZZ/2\ZZ$-graded space $V=V_{\bar 0}\oplus V_{\bar 1}$ is called a {\em superspace}. A {\em superalgebra} is an algebra whose underlying space is a superspace.  For a homogeneous element $v\in V_{\bar i}$, we denote by $|v|$ its parity, i.e. $|v|=\bar i$.

Below, every superalgebra identity is assumed to involve only homogeneous elements.

\subsection{Linearly compact superalgebras} Recall that a topological vector space or superspace is {\em linearly compact} if it is isomorphic to a topological product of finite-dimensional vector spaces endowed with discrete topology.  A topological algebra is  linearly compact if its underlying vector space is linearly compact.

In particular, let  $A$ be an algebra with a decreasing filtration $A=A_k\supset A_{k-1}\supset\dots$ such that $\dim A_i/A_{i-1}<\infty$ for all $i$.  Then the system of fundamental neighborhoods $\{A_i\}$ defines a topology on $A$.  The completion of $A$ in this topology is a linearly compact algebra.  For example, the  superalgebra $\wedge(m,n)=\wedge(n)[[x_1,\dots,x_m]]$ is linearly compact.  Here and further $\wedge(n)$ denotes the Grassmann algebra in $n$ (anticommuting) indeterminates.

\subsection{Jordan superalgebras} A {\em Jordan superalgebra} $J$ is an $\FF$-superalgebra such that for every homogeneous $a,b,c,d\in J$, the product $\circ$ satisfies
\begin{align*}
&a\circ b=(-1)^{|a||b|}b\circ a\qquad \text{(commutativity)} \\
&(-1)^{|a||c|}(a\circ b)\circ(c\circ d)+(-1)^{|a||b|}(b\circ c)\circ (a\circ d)\\
&+(-1)^{|b||c|}(c\circ a)\circ (b\circ d)=(-1)^{|a||c|}a\circ((b\circ c)\circ d)\\&+(-1)^{|a||b|}b\circ((c\circ a)\circ d)+(-1)^{|b||c|}c\circ((a\circ b)\circ d).\\
&\qquad\qquad\qquad\qquad\qquad\text{(linearized Jordan identity)}
\end{align*}

\subsection{KKM doubles}  With few exceptions (see below) linearly compact simple Jordan superalgebras can be constructed from generalized Poisson superalgebras.

A {\em generalized Poisson superalgebra} $A$ is a vector superspace endowed with an associative commutative product $(a,b)\mapsto ab$ and a Lie superalgebra bracket $(a,b)\mapsto\{a,b\}$ such that
\begin{equation*}
\{a,bc\}=\{a,b\}c+(-1)^{|a||b|}b\{a,c\}-D(a)bc
\end{equation*}
for some even derivation $D$.

Let $A$ be a commutative associative superalgebra with a (possibly zero) even derivation $D$.  A {\em Jordan bracket} on $A$ is a bilinear operation $\{\ ,\ \}: A\otimes A\to A$ such that

(i) $\{a,b\}=-(-1)^{|a||b|}\{b,a\}$;

(ii) $\{a,bc\}=\{a,b\}c+(-1)^{|a||b|}b\{a,c\}-D(a)bc$;

(iii) $\{\{a,b\}c\}+(-1)^{|a||b|+|a||c|}\{\{b,c\},a\}+(-1)^{|a||c|+|b||c|}\{\{c,a\},b\}=-\{a,b\}D(c)+(-1)^{|a||b|+|a||c|}\{b,c\}D(a)+(-1)^{|a||c|+|b||c|}\{c,a\}D(b)$.

Clearly $D(a)=\{a,1\}$.  When $D\equiv 0$, we get an ordinary Possion bracket.

\begin{lemma}\cite{CKj} Let $A$ be a generalized Poisson superalgebra.  Then the bracket
\begin{equation*}
\{a,b\}_D=\{a,b\}+\frac{1}{2}(aD(b)-D(a)b)
\end{equation*}
is Jordan with respect to the derivation $D/2$.
\end{lemma}

Let $A$ be a generalized Poisson superalgebra and
consider the vector superspace $K(A)=A\oplus A\theta$, where $\theta$ is a formal variable with parity $\bar 1$.  Introduce the following product $\circ$ on $K(A)$:
\begin{equation*}
a\circ b=ab; \quad a\circ (b\theta)=(ab)\theta; \quad (a\theta)\circ b=(-1)^{|b|} (ab)\theta; \quad (a\theta)\circ(b\theta)=(-1)^{|b|}\{a,b\}_D.
\end{equation*}
The superalgebra $(K(A),\circ)$ is Jordan.  It is called the {\em Kantor--King--McCrimmon} (KKM) {\em double}. 

\begin{lemma}\cite{KM} Let $A$ be a unital generalized Poisson superalgebra.  Then $K(A)$ is simple iff $A$ is simple.
\end{lemma}

Let $A$ be a unital generalized Poisson superalgebra.  Let $\d$ be a derivation of $A$ that commutes with $D$.  Then $\d$ can be extended to $K(A)$ by putting $\d(a\theta)=(\d a)\theta$. This is clearly a derivation of the superalgebra $K(A)$.  Another example of a derivation of $K(A)$ is the derivation $\d_\theta$ defined as 
\begin{equation}\label{eq.dtheta}
\d_\theta(A)=0 \text { and } \d_\theta(\theta)=1.  
\end{equation}
More generally, $\Der K(A)$ contains a subalgebra $A\d_\theta$.  Therefore, we conclude that $\Der K(A)$ contains a subalgebra isomorphic to $\Der A\rtimes A$.

In principle, even the KKM double of a (non-generalized) Poisson algebra can possess derivations not in $\Der A\rtimes A$ as the following example demonstrates.

\begin{example} Let $A=\FF\oplus\FF\xi$, $\{\xi,\xi\}=1$, where $\xi$ is odd.  A direct computation shows that $\Der A$ contains no even derivations.  However, $\Der K(A)$ has a derivation $\delta$ such that $\delta(\theta)=\theta$, $\delta(\xi)=-\xi$.  
\end{example}

We see that the relation between $\Der K(A)$ and $\Der A$ is quite complicated and computing $\Der K(A)$ should not be attempted through simply lifting $\Der A$ to the KKM double.  For another approach, see Corollary~\ref{derjmn}.

\subsection{Examples of KKM doubles}  Consider the commutative associative superalgebra $\wedge(m,n)$ with odd indeterminates $\xi_1,\dots,\xi_n$ and even indeterminates $p_1,\dots,p_k,q_1,\dots,q_k$ for $m=2k$ and $p_1,\dots,p_k,q_1,\dots,q_k,t$ for $m=2k+1$.

In the case $m=2k$, $\wedge(m,n)$ carries the Poisson bracket
\begin{equation}\label{pbracket}
\{f,g\}=\sum_{i=1}^k\left(\frac{\partial f}{\partial p_i}\frac{\partial
g}{\partial q_i}-\frac{\partial f}{\partial q_i}\frac{\partial
g}{\partial p_i}\right)+(-1)^{|f|}\left(\sum_{i=1}^{n-2}\frac{\partial f}{\partial \xi_i}
\frac{\partial g}{\partial \xi_i}+
\frac{\partial f}{\partial \xi_{n-1}}
\frac{\partial g}{\partial \xi_n}+
\frac{\partial f}{\partial \xi_n}
\frac{\partial g}{\partial \xi_{n-1}}\right),
\end{equation}
and in the case $m=2k+1$, the generalized Poisson bracket
\begin{multline}\label{kbracket}
\{f,g\}=(2-E)f\frac{\partial g}{\partial t}-\frac{\partial f}{\partial t}(2-E)g+\\
\sum_{i=1}^k\left(\frac{\partial f}{\partial p_i}\frac{\partial
g}{\partial q_i}-\frac{\partial f}{\partial q_i}\frac{\partial
g}{\partial p_i}\right)+(-1)^{|f|}\left(\sum_{i=1}^{n-2}\frac{\partial f}{\partial \xi_i}
\frac{\partial g}{\partial \xi_i}+
\frac{\partial f}{\partial \xi_{n-1}}
\frac{\partial g}{\partial \xi_n}+
\frac{\partial f}{\partial \xi_n}
\frac{\partial g}{\partial \xi_{n-1}}\right),
\end{multline}
where $E$ is the Euler operator
\begin{equation*}
E=\sum_{i=1}^k\left(p_i\frac{\partial}{\partial
p_i}+q_i\frac{\partial}{\partial q_i}\right)+\sum_{i=1}^n\xi_i
\frac{\partial}{\partial \xi_i}.
\end{equation*}

In both cases, we denote the resulting (generalized) Poisson superalgebra as $P(m,n)$.  Its KKM double is the Jordan superalgebra, denoted by $J(m,n)$.  Clearly, $J(m,n)$ is linearly compact.

For our purposes, we need to compute $\Der J(m,n)$.  However, using the KKM double construction to relate $\Der J(m,n)$ to $\Der P(m,n)$  is difficult  (see the discussion above).  We therefore delay this computation and derive it from Theorem~\ref{tkkprop} below.

\begin{remark} With respect to the bracket $\{\ ,\ \}$, $P(m,n)$ can be considered as a Lie superalgebra.  For $m$ even, it has a one-dimensional center.  Let $H'(m,n)=P(m,n)/\FF 1$, and let $H(m,n)=[H'(m,n),H'(m,n)]$.  $H(m,n)$ is a simple Lie superalgebra. 

For $m$ odd, $P(m,n)$ is a simple Lie superalgebra itself.  To separate the Poisson and Lie context, we denote it as $K(m,n)$.

For a discussion of $H(m,n)$ and $K(m,n)$, see \cite{Ks}.
\end{remark}

\subsection{One other example}  Consider the commutative associative superalgebra $\wedge(1,1)$ with a derivation $D=\frac{\d}{\d\xi}+\xi\frac{\d}{\d t}$.  Denote by $\vee$ the space $\wedge(1,1)$ with reversed parity.  Define the product $\circ$ on $\vee$ as $a\circ b=aD(b)+(-1)^{|b|}D(a)b$.  This makes $\vee$ into a Jordan superalgebra that we denote $\js(1,1)$.  For future use, we remark that $\js(1,1)$ also inherits the topology of $\wedge(1,1)$.

\begin{lemma}\label{der.js} The set of even surjective derivations of $\js(1,1)$ is $\FF\frac{\d}{\d t}$.
\end{lemma}

\begin{proof} Let $\d$ be a surjective derivation. Since $\js(1,1)_{\bar 0}\simeq\FF[\xi t]$, the restriction of $\d$ to $\js(1,1)_{\bar 0}$ is proportional to $\frac{\d}{\d t}$.

Because $\xi t^k\circ t^0=t^k$ and $\js(1,1)_{\bar 1}$ is spanned by the powers of $t$, $\js(1,1)_{\bar 1}=\js(1,1)_{\bar 0}\circ t^0$.  The action of $\d$ on $\js(1,1)$ is thus completely determined by its restriction to $\js(1,1)_{\bar 0}$ and $\d(t^0)$.  By computing $\d(t^1\circ t^0)$, we see that $\d(t^0)$ must be zero for any $\d$.
\end{proof}

We postpone the discussion of one other Jordan superalgebra, $\jck(1,4)$, until Example~\ref{ex.jck}.

\subsection{TKK construction}  A Lie superalgebra $L$ possesses a {\em short grading} if $L=L_{-1}\oplus L_0\oplus L_1$.  If $L$ contains an $\sli_2$-triple $\{f,e,h\}$ such that $f\in L_{-1}, h\in L_0, e\in L_1$, the copy of $\sli_2(\FF)$ spanned by this triple is called a {\em short subalgebra}.  

Existence of a short subalgebra assures that $L_{-1}$ possesses the structure of a Jordan superalgebra.  Specifically, $a\circ b$ is defined as $[[a,e],b]$.

For every unital Jordan superalgebra $J$ there exists a Lie superalgebra $L(J)$ with short grading such that $L(J)_{-1}\simeq J$.  The specific construction of $L(J)$ (the Tits--Kantor--Koecher or TKK construction) can be found in \cite{Jac} (see also \cite{CKj}).  The TKK construction can be also extended to non-unital Jordan algebras by adjoining the identity first and then considering the Lie algebra $L(J\oplus\FF)$.

\begin{example}\cite{CKj}  For the Jordan superalgebras $J(m,n)$,
\begin{equation*}
L(J(m,n))=
 \begin{cases}
  H(m,n+3),\ m\text{ even},\\
  K(m,n+3),\ m\text{ odd}.
 \end{cases}
\end{equation*}
The corresponding short subalgebra is spanned by the $\sli_2$-triples
\begin{align*}
&\xi_{n+1}\xi_{n+2}, \xi_{n+2}\xi_{n+3},\xi_{n+1}\xi_{n+3},\qquad m\text{ even},\\
&\xi_{n+1}\xi_{n+2}, \xi_{n+3}\xi_{n+2}, \xi_{n+1}\xi_{n+3},\qquad m\text{ odd}.
\end{align*}

Specifically, the embedding $J(m,n)\to L(J(m,n))$ is given by 
\begin{equation}\label{eq.tkkforkkm}
a+b\theta\mapsto (a\xi_{n+1}+b)\xi_{n+3},\qquad a,b\in\wedge(m,n).
\end{equation}
\end{example}

Statements collected as Theorem~\ref{tkkprop}(a)-(e) can be found in or easily deduced from \cite[Sections VII.5-6]{Jac}.  Part (f) follows directly from the definitions.

\begin{theorem}\label{tkkprop}  Let $J$ be a Jordan superalgebra and $L(J)$ the corresponding Lie  superalgebra obtained via the TKK construction.  Then

(a) every ideal of $L(J)$ is graded with respect to the standard short grading;

(b) $J$ is simple iff $L(J)$ is simple;

(c) the action of $\Der J$ extends to $L(J)$;

(d) for $J$ unital, $\Der J=\{\phi\,|\,\phi\in\Der L(J), \phi|_{\sli_2}=0\}$.

(e) $J$ is differentiably simple iff $L(J)$ differentiably simple;

(f) $J$ is linearly compact iff $L(J)$ is linearly compact.
\end{theorem}

We are going to apply Theorem~\ref{tkkprop}(d) to compute $\Der J(m,n)$.  Recall that $\Der K(m,n)=K(m,n)$ and $\Der H(m,n)=CH(m,n)$, provided that $m\geq 1$ \cite[Proposition 1.8]{CKl}.  Here $$CH(m,n)=H(m,n)\oplus\FF\left(2\sum_{i=1}^k\left(p_i\frac{\partial}{\partial p_i}+q_i\frac{\partial}{\partial q_i}\right)+\sum_{i=1}^n\xi_i\frac{\partial}{\partial\xi_i}\right).$$  

In the cases of $m$ even, respectively odd, we need to calculate which elements of $CH(m,n+3)$, respectively $K(m,n+3)$, act trivially on the triple $T=(\xi_{n+1}\xi_{n+2},\linebreak \xi_{n+2}\xi_{n+3},\xi_{n+1}\xi_{n+3})$.  Consider an element $D$ of $H(m,n)$, respectively $K(m,n)$, $D=a+b\xi_{n+1}+c\xi_{n+2}+d\xi_{n+3}+e\xi_{n+1}\xi_{n+2}+f\xi_{n+1}\xi_{n+3}+g\xi_{n+2}\xi_{n+3}+h\xi_{n+1}\xi_{n+2}\xi_{n+3}$, where $a,b,c,d,e,f,g,h\in K(m,n)$, respectively $H(m,n)$.  A direct computation utilizing~(\ref{kbracket}) and~(\ref{pbracket}) shows that if $\ad D$ kills $T$, $D=a+h\xi_{n+1}\xi_{n+2}\xi_{n+3}$.  Note that $\xi_{n+1}\xi_{n+2}\xi_{n+3}=\d_\theta$, where $\d_\theta$ is defined as in (\ref{eq.dtheta}).  When $m$ is even, the central element of $CH(m,n)$ also acts trivially on the short subalgebra.

Recall that the superalgebra $J(m,n)$ is the KKM double of the (generalized) Poisson algebra $P(m,n)$.
We thus have the following

\begin{corollary}\label{derjmn} For the Jordan superalgebras $J(m,n)$, $m\neq 0$,
\begin{equation*}
\Der J(m,n)=
 \begin{cases}
  CH(m,n)\oplus P(m,n)\d_\theta,\ m\text{ even},\\
  K(m,n)\oplus P(m,n)\d_\theta,\ m\text{ odd}.
 \end{cases}
\end{equation*}
\end{corollary}

The particular embedding (\ref{eq.tkkforkkm}) of $J(m,n)$ into its TKK Lie superalgebra $L(J(m,n))$ as well as the description of $\Der J(m,n)\subset \Der L(J(m,n))$ above imply  $\Der K(P(m,n))=\Der P(m,n)\rtimes P(m,n)$.  

\section{Preliminaries: Pseudoalgebras}\label{sec2}

\subsection{Main definitions} Let $H$ be a cocommutative Hopf $\FF$-algebra.  Recall that $H\otimes H$ is a right $H$-module with an action $(a\otimes b)c=ac_{(1)}\otimes bc_{(2)}$, where $c_{(1)}\otimes c_{(2)}=\Delta(c)$ in Sweedler's notation.  (As usual, $\otimes$ stands for $\otimes_\FF$.)  

An  {\em $H$-pseudoalgebra} $R$ is a left $H$-module with an operation
\begin{equation*}
*: R\otimes R\to (H\otimes H)\otimes_H R
\end{equation*}
such that for $a,b\in R$ and $f,g\in H$,
\begin{equation*}
fa*gb=((f\otimes g)\otimes_H 1)(a*b)\qquad\text{(bilinearity)}.
\end{equation*}

For a general discussion of pseudoalgebras, see \cite[Chapter 3]{BDK}.

Pseudoproducts of m elements of $R$ live in $H^{\otimes m}\otimes_H R$ with the action of $H$ on $H^{\otimes m}$ defined in accordance with the bracketing scheme.  For example, $(a*b)*c=\sum_{i,j}(f_i{f_{ij}}_{(1)}\otimes g_i{f_{ij}}_{(2)}\otimes g_{ij})\otimes_H e_{ij}$, where $a*b=\sum_i (f_i\otimes g_i)\otimes_H e_i$ and $e_i*c=\sum_{i,j} (f_{ij}\otimes g_{ij})\otimes_H e_{ij}$.

We call an $H$-pseudoalgebra {\em finite} if it is of finite rank as an $H$-module.

From this point on, we will always assume that the underlying $H$-module of a pseudoalgebra is a superspace, i.e. that $R=R_{\bar 0}\oplus R_{\bar 1}$, where both $R_{\bar 0}$ and $R_{\bar 1}$ are preserved by $H$.

\begin{example} A {\em conformal superalgebra} $R$ \cite{FK} is a pseudoalgebra over the Hopf algebra $H=\FF[\d]$.  A pseudoproduct in a conformal superalgebra has the form $a*b=\sum (\d^i\otimes \d^j)\otimes_H c_{ij}$.  Using the comultiplication action of $H$ on $H\otimes H$, the above product can always be rewritten as $a*b=\sum \left((-\d)^k\otimes 1\right)\otimes_H c_k$.  In such formulation the collection $\{c_k\}$ is unique, thus $a*b$ can be viewed as an element of $H\otimes R$. 

Let $\lambda=-\d\otimes 1$ and define $a_\lambda b=\sum \lambda^k c_k\in \FF[\lambda]\otimes R\simeq H\otimes R$.

This allows for a definition of a conformal superalgebra independent of the pseudoalgebra formalism: a conformal superalgebra $R$ is a $\FF[\d]$-module with a $\lambda$-product $a_\lambda b\in\FF[\lambda]\otimes R$ such that 
\begin{equation*}
(\d a)_\lambda b=-\lambda a_\lambda b,\qquad a_\lambda (\d b)=(\lambda+\d) a_\lambda b.
\end{equation*}

Every conformal algebra can be constructed as an algebra of {\em formal distributions} over an ordinary algebra.  Namely, let $A$ be an algebra.  Introduce the $\lambda$-product on the set $A[[z,z^{-1}]]$ of formal distributions: $a(z)_\lambda b(z)={\rm Res}|_{w=0}a(z)b(w)e^{(\lambda(z-w))}$.  The subset of formal distributions whose $\lambda$-products are polynomial in $\lambda$ form a conformal algebra.

\end{example}

As with ordinary algebras, one can consider varieties of pseudoalgebras defined by identities.  Introduce the following identities for pseudoalgebras:
\begin{align*}
& a*b=(-1)^{|a||b|}(\mathtt{(12)}\otimes_H\id)(b\otimes a),\qquad\text{(commutativity)}\\
& a*b=-(-1)^{|a||b|}(\mathtt{(12)}\otimes_H\id)(b\otimes a),\qquad\text{(anti-commutativity)}\\
& (a*b)*c=a*(b*c),\qquad\text{(associativity)}\\
& a*(b*c)=(a*b)*c+(-1)^{|a||b|}(\mathtt{(12)}\otimes\id)(b*(a*c)),\qquad\text{(Jacobi identity)}\\
& (-1)^{|a||c|}(a* b)*(c* d)+(-1)^{|a||b|}(\mathtt{(123)}\otimes_H\id)((b* c)* (a* d))\\&\quad +(-1)^{|b||c|}(\mathtt{(132)}\otimes_H\id)((c* a)* (b* d))\\&=(-1)^{|a||c|}a*((b* c)* d)+(-1)^{|a||b|}(\mathtt{(123)}\otimes_H\id)(b*((c* a)* d))\\&\quad +(-1)^{|b||c|}(\mathtt{(132)}\otimes_H\id)(c*((a* b)* d))).\qquad\text{(Jordan identity)}
\end{align*}
Here $\mathtt{(ij)}$ and $\mathtt{(ijk)}$ denote cyclic permutations of respective components of  $H^{\otimes m}$.

A pseudoalgebra is {\em associative} if it satisfies the associativity identity, {\em Lie} if it satisfies the anti-commutativity and Jacobi identity, {\em Jordan} if it satisfies the commutativity and Jordan identity, etc.

For the $\lambda$-product in conformal algebras, the above identities become
\begin{align*}
& a_\lambda b=(-1)^{|a||b|}b_{-\d-\lambda} a,\qquad\text{(commutativity)}\\
& a_\lambda b=-(-1)^{|a||b|}b_{-\d-\lambda} a,\qquad\text{(anti-commutativity)}\\
& a_\lambda(b_\mu c)=(a_\lambda b)_{\lambda+\mu} c,\qquad\text{(associativity)}\\
& a_\lambda(b_\mu c)=(a_\lambda b)_{\lambda+\mu} c+(-1)^{|b||c|}b_\mu(a_\lambda c),\qquad\text{(Jacobi identity)}\\
&  (-1)^{|a||c|}a_\lambda ((b_\mu c)_\nu d)+
(-1)^{|a||b|}b_\mu((c_{\nu-\lambda}a)_{\lambda+\nu} d)\\&\quad+
(-1)^{|b||c|}c_{\nu-\mu}((a_{-\mu-\d} b)_{\lambda+\nu} d)\\
&=(-1)^{|a||c|}(a_{-\mu-\d}b)_{\lambda+\mu} (c_{\nu-\mu}d)+
(-1)^{|a||b|}(b_\mu c)_\nu (a_{\lambda+\nu} d)\\&\quad+
(-1)^{|b||c|}(c_{\nu-\mu}a)_{\lambda+\nu-\mu} (b_\mu d).\qquad\text{(Jordan identity)}
\end{align*}

\begin{example}\label{ex.current}  Let $\fa$ be an $\FF$-superalgebra.  Then $H\otimes\fa$ is an $H$-pseudoalgebra with the pseudoproduct defined by
\begin{equation*}
(1\otimes a)*(1\otimes b)=(1\otimes 1)\otimes_H (ab),\quad a,b\in\fa
\end{equation*}
and extended to all of $H\otimes\fa$ by bilinearity.  $H\otimes\fa$ is called a {\em current} pseudoalgebra and is denoted $\cur\fa$.

In the conformal setting, the current conformal superalgebra over $\fa$ is $\FF[\d]\otimes_{\FF}\fa$ with $a_\lambda b=ab$ for $a,b\in\fa$.

More generally, consider a Hopf subalgebra $H'$ of $H$ and an $H'$-pseudoalgebra $A$.  Then $H\otimes_{H'} A$ can be endowed with an $H$-pseudoproduct
\begin{equation*}
(1\otimes_{H'}a)*(1\otimes_{H'}b)=\left((1\otimes 1)\otimes_H 1\right)(a*b),\quad a,b\in A
\end{equation*}
extended to all of $H\otimes_{H'} A$ by bilinearity.  The $H$-pseudoalgebra $H\otimes_{H'} A$ is denoted $\cur^H_{H'} A$.
\end{example}

A subalgebra $I$ of a pseudoalgebra $R$ is an {\em ideal} if $a*b\in (H\otimes H)\otimes_H I$ whenever $a\in I$ or $b\in I$.  A pseudoalgebra is {\em simple} if it has no nontrivial proper ideals.  Other algebraic concepts can be as easily carried to the pseudoalgebra setting, see \cite{BDK} for details.

A cocommutative Hopf algebra is always isomorphic to a smash product of a group algebra and a universal enveloping algebra of some Lie algebra $\fd$: $H\simeq U(\fd)\sharp\FF(\Gamma)$.  When $\Gamma$ is finite, the theory of $H$-pseudoalgebras is essentially that of $U(\fd)$-pseudoalgebras with a $\Gamma$-action \cite[Corollary 5.1]{BDK}.  Thus we will concentrate on the case of $H=U(\fd)$, where $\fd$ is a finite-dimensional Lie algebra.

\subsection{Annihilation algebras}  Let $H$ be a cocommutative Hopf algebra.  The dual algebra $H^*=\Hom_\FF(H,\FF)$ carries the structure of a left and right $H$-module:
\begin{equation*}
\langle hx,f\rangle=\langle x,S(h)f\rangle,\quad \langle xh,f\rangle=\langle x,fS(h)\rangle\qquad\text{for } h,f\in H, x\in H^*.
\end{equation*} 
Let $R$ be an $H$-pseudoalgebra.  Consider the space $\coef(R)=H^*\otimes_H R$. $\coef(R)$ inherits the left $H$-module structure from $H^*$,
\begin{equation*}
h(x\otimes_H a)=(hx)\otimes_H a,
\end{equation*}
and has a natural product
\begin{align*}
&(x\otimes_H a)(y\otimes_H b)=\sum_i (xf_i)(yg_i)\otimes_H e_i,\\
&\qquad \text{where } a*b=\sum_i (f_i\otimes g_i)\otimes_H e_i.
\end{align*}

Moreover, if $R$ is an associative, Lie, or Jordan pseudoalgebra, then $\coef(R)$ is, respectively, an associative, Lie, or Jordan superalgebra.

$\coef(R)$ is called the {\em annihilation algebra} of $R$. 

\begin{example} $\coef(\cur\fa)=H^*\otimes \fa$.  More generally, $\coef(\cur^H_{H'} A)=H^*\otimes_{H'}A$.
\end{example}

\begin{lemma}\label{coeffsimple} Let $J$ be a simple Jordan pseudoalgebra.  Then $\coef(J)$ contains no $H$-stable ideals.
\end{lemma}

\begin{proof}  See the proof of \cite[Proposition 3.13]{BDK}.
\end{proof}

We also remark that for a finite pseudoalgebra $J$, $\coef(J)$ is linearly compact by construction.

\subsection{TKK construction for pseudoalgebras}  Similarly to Jordan superalgebras, finite Jordan pseudoalgebrs can be embedded into short-graded Lie pseudoalgebras.  Specifically, let $J$ be a finite Jordan pseudoalgebra.  Then there exists a Lie pseudoalgebra $L=L_{-1}\oplus L_0\oplus L_1$ such that $L_{-1}\simeq J$ as $H$-modules. 

The construction of $L(J)$ is given in \cite[Section 4]{Kol} in the non-super case and can easily be extended to the superalgebras.  In this paper, however, we do not require the specifics of the pseudoalgebra TKK construction and thus omit the details.  Namely, given a finite Jordan pseudoalgebra $J$, we have
\begin{equation*}
\begin{CD}
J @>\phantom{longer}\text{TKK}\phantom{longer}>> L(J)\\
@V\text{annihilation}VV @VV\text{annihilation}V\\
\coef(J) @>\text{TKK}>> L(\coef(J))\simeq \coef(L(J))
\end{CD}.
\end{equation*}

The canonical isomorphism in the lower right corner of the diagram follows from commutativity of the TKK and $\coef$ functors, which preserve, respectively, the $H$-action and the short grading.

To classify simple finite Jordan pseudoalgebras one should start with either the classification of their annihilation algebras or their Lie pseudoalgebra counterparts (respectively, go up from the bottom left or left from the top right corners of the diagram).  We utilize the former approach because the list of potential annihilation algebras exists already \cite{CKj}, whereas all simple Lie pseudoalgebras are unknown at the moment.  Note that in the non-super case it was possible to use the classification of simple Lie non-super pseudolagebras $L(J)$ (i.e. move along the top row of the diagram) \cite{Kol}.  Both approaches ultimately rely on our understanding of linerly compact Lie superalgebras and their derivations (the lower right corner of the diagram).

\subsection{Reconstruction}  Given a linearly compact superalgebra $\mj$ with an action of $\fd$, one wishes to find all corresponding pseudolagebras $J$ such that $\coef(J)=\mj$.  This is known as a reconstruction problem.  The reconstruction is straightforward  for one-dimensional $\fd$ (i.e. for conformal superalgebras) but is more complicated in higher dimensions.  Here we concentrate on the case of conformal superalgebras.

\begin{lemma} Let $J$ be a Jordan conformal superalgebra.  Then $\Tor_{\FF[\d]} J$ is an ideal of $J$.
\end{lemma}

\begin{proof}\cite{DK} Let $a\in Tor J$, i.e. such that $p(\d)a=0$ for some $p(\d)\in\FF[\d]$.  For any $b\in J$, $0=\left(p(\d)a\right)_\lambda b=p(-\lambda)(a_\lambda b)$, hence $a_\lambda b=0$.
\end{proof}

It follows that a simple Jordan conformal algebra is torsion-free, hence free as a $\FF[\d]$-module.  

\begin{corollary}\label{confalg.reconst}  A simple finite Jordan conformal superalgebra $J$ is uniquely determined by its annihilation algebra $\coef(J)$ and the $\FF[\d]$-action on $\coef(J)$.
\end{corollary}

\begin{proof} Follows directly from \cite[Proposition 5.1]{DK}.
\end{proof}

\subsection{Conformal algebras from KKM doubles}  Consider the Jordan superalgebra $J(1,n)=\wedge(1,n)\oplus\wedge(1,n)\theta$.  For every $a\in\wedge(n)$, consider the formal distributions
\begin{equation*} 
a^+=\sum_{n\in\ZZ} (at^n) z^{-n-1}\quad\text{ and }\quad a^-=\sum_{n\in\ZZ} (at^n\theta) z^{-n-1}.  
\end{equation*}
Let $a$ and $b$ be homogeneous elements of $\wedge(n)$ of degrees $r$ and $s$, respectively.  We have the following $\lambda$-products: 
\begin{align*}
& {a^+}_\lambda b^\pm=(ab)^\pm\\
& {a^-}_\lambda {b^+}=(-1)^s(ab)^-\\
& {a^-}_\lambda {b^-}= 
(-1)^s\left((r-1)\d(ab)^+ +(-1)^r\left(\sum_{i=1}^{n-2} \d_ia\d_ib+ \d_{n-1}a\d_nb+\d_na\d_{n-1}b\right)^+\right.\\ &\left.\phantom{{a^-}_\lambda {b^-}\sum_a^b}+\lambda(r+s-2)(ab)^+\right).
\end{align*}
Thus the formal distributions above span a conformal algebra.  We denote it as $J_n$.

\begin{remark} The last formula above is consistent with the $\lambda$-bracket for the conformal algebra $K_n$ \cite[Example 3.8]{FK} (modulo the different ways of writing brackets for $K(1,n)$ here and in \cite{FK}).
\end{remark}

\subsection{Further examples of pseudoalgebras} We present here two conformal superalgebras (i.e. pseudoalgebras over $\FF[\d]$) whose annihilation algebras are the exceptional conformal superalgebras $\js(1,1)$ and $\jck(1,4)$.

\begin{example}\label{ex.js} The conformal superalgebra $\js_1$ is freely spanned over $\FF[\d]$ by an even element $S$ and an odd element $T$ such that 
\begin{equation*}
S_\lambda S=2S, \quad T_\lambda T=(\d+2\lambda)S,\quad T_\lambda S=T.
\end{equation*}
 The annihilation algebra of $\js_1$ is $\js(1,1)$ with the surjective derivation $\frac{\d}{\d t}$.  Written in the form of formal distributions $S$ and $T$ are $S=\sum (\xi t^n)z^{-n-1}$ and $T=\sum t^nz^{-n-1}$ (with reversed parities).
\end{example}

\begin{example}\label{ex.jck} The conformal superalgebra $\jck_4$ is freely spanned over $\FF[\d]$ by four even and four odd elements.  Its annihilation algebra is the exceptional Jordan superalgebra $\jck(1,4)$ whose TKK Lie algebra is $E(1,6)\subset K(1,6)$.  

Below we explicitly describe embeddings of $\jck_4$ into Lie conformal superalgebras $K_6$ and $CK_6$ and the Jordan conformal superalgebra $J_3$.

Consider the Lie conformal superalgebra $K_6$ with the annihilation algebra $K(1,6)$.  As a $\FF[\d]$-module, $K_6$ is freely spanned by $\wedge(6)$.  Introduce the following change of basis in $\wedge(6)$: instead of odd indeterminates $\xi_i$, use $\omega_i$ defined as $\omega_i=\xi_i$, $i=1,\dots, 4$, $\omega_5=\epsilon\xi_5+\alpha\epsilon\xi_6$, $\omega_6=\alpha\epsilon\xi_5+\epsilon\xi_6$, where $\alpha,\epsilon\in\FF$, $\alpha^2=-1$ and $\epsilon^2=\alpha/2$.  (This is essentially the basis used in \cite{CK6}.)  The $\lambda$-bracket on $K_6$ then becomes
\begin{equation*}
[a_\lambda b]=\left(\left(\frac{r}{2}-1\right)\d(ab)+(-1)^r\frac{1}{2}\sum_{i=1}^6\delta_ia\delta_ib\right)+\lambda\left(\frac{r+s}{2}-2\right)ab,
\end{equation*}
where $a=\omega_{i_1}\dots\omega_{i_r}$, $b=\omega_{j_1}\dots\omega_{j_s}$, and $\delta_i=\d/\d\omega_i$.

$K_6$ contains a simple subalgebra $CK_6$ spanned over $\FF[\d]$ by the elements 
\begin{align*}
&L=(-1+\alpha\d^3\nu)/2,\qquad a_{ij}=\omega_i\omega_j+\alpha\d(\omega_i\omega_j)^*,\\ 
&b_i=\omega_i-\alpha\d^2(\omega_i)^*,\quad c_{ijk}=\omega_i\omega_j\omega_k+\alpha(\omega_i\omega_j\omega_k)^*, 
\end{align*}
where $\nu=\omega_1\dots\omega_6$ and $a^*\in\wedge(6)$ is a monomial in $\omega_i$'s such that $aa^*=\nu$.  The Lie superalgebra $E(1,6)$ is the annihilation algebra of $CK_6$.  A short grading of $E(1,6)$ gives rise to $\jck(1,4)$.  Specifically, $\jck(1,4)$ is the $-1$-eigenspace of $\xi_6\xi_5$ or, equivalently, the $-\alpha$-eigenspace of $\omega_5\omega_6$.  Since the short grading is consistent with the conformal structure, the $-\alpha$ eigenspace of $\ad\omega_5\omega_6$ acting on $CK_6$ is a Jordan conformal superalgebra.  We denote it $\jck_4$.

$\jck_4$ is spanned by even elements $a_{i6}-\alpha a_{i5}$, $i=1,\dots,4$, and odd elements $b_5-\alpha b_6$, $c_{126}-c_{346}, c_{136}+c_{246}, c_{236}-c_{146}$.  The $\lambda$-product in $\jck_4$ is defined, as in the TKK construction, via the bracket in $K_6$:  $$a_\lambda b=\left.\left[\left[a_\mu (\xi_4\xi_5)\right]_\lambda b\right]\right|_{\mu=0}.$$

Recall that the embedding  $\phi:E(1,6)\hookrightarrow K(1,6)$ respects the short grading
of both algebras with respect to the same short subalgebra.  Restricting $\phi$
to the $-1$st components, we obtain the embedding of $\jck(1,4)$ into $J(1,3)$.
Just as for $CK_6$ and $K_6$, this embedding passes to the level of conformal algebras,
yielding $\jck_4\hookrightarrow J_3$.  Specifically, recomputing the basis of 
$\jck_4$ obtained above in terms of $\xi_i$'s and taking into account (\ref{eq.tkkforkkm}),
we see that $JCK_4$ is a subalgebra of $J_3$ spanned over $\FF[\d]$ by elements
\begin{align*}
\xi_i^--\d(\xi_i^*)^+,\quad 1^++\d\nu^-,\qquad \xi_i^++(\xi_i^*)^-,\quad 1^--\d^2\nu^+,
\end{align*}
where $i=1,2,3$, $\nu=\xi_1\xi_2\xi_3$ and $\xi_i^*$ is a monomial in $\xi_i$'s such that $\xi_i\xi_i^*=\nu$.
\end{example}

Since $\Der E(1,6)\subset \Der K(1,6)$ and the TKK construction of $\jck(1,4)$ and $J(1,3)$ utilize the same short subalgebra, Theorem~\ref{tkkprop}(d) implies that $\Der \jck(1,4)\subset\Der J(1,3)$.  By Corollary~\ref{derjmn}, $\Der J(1,3)=K(1,3)\oplus P(1,3)\d_\theta$.  Since $\d_\theta$ does not preserve $\jck(1,4)$, we have the following

\begin{corollary}\label{derjck} $\Der \jck(1,4)\subset K(1,3)$.
\end{corollary}

\section{Simple Jordan Pseudoalgebras}\label{sec3}

Throughout this section all pseudoalgebras and conformal superalgebras are assumed to be finite.

\subsection{Annihilation algebra}  The following two statements are slight extensions of \cite[Theorem 13.1]{BDK} and rely heavily on its proof.  By $\co_r$ we denote the algebra $\FF[[t_1,\ldots,t_r]]$ and $\fd$ is always taken to be a finite-dimensional Lie algebra.

\begin{proposition}\label{13.7.lie} If $L$ is a finite simple Lie $H=U(\fd)$-pseudoalgebra, then as a topological Lie superalgebra, $\coef(L)$ is isomorphic to an irreducible central extension of a current Lie superalgebra $\co_r\otimes\fs$, where $\fs$ is a simple linearly compact Lie superalgebra of growth $\dim\fd-r$.
\end{proposition}

\begin{proof}  It follows from the proof of \cite[Theorem 13.1]{BDK} and the super-version of the Cartan--Guillemin theorem that $\coef(L)$ is an irreducible central extension of $\co_r\otimes\wedge(m)\otimes\fs$, where $\fs$ is a simple linearly compact Lie superalgebra.  In order to complete the proof along the lines of \cite[Theorem 13.1]{BDK}, it remains to show that $m=0$.  If $m>0$, consider the ideal $\co_r\otimes\wedge'(m)\otimes\fs$ (here $\wedge'(m)$ is the augmented ideal of $\wedge(m)$).  It is regular: $\fd$ acts by even derivations which have the form
$D_0\otimes 1+\sum f_i\otimes D_i$, where $D_0$ is an even derivation of $\co_r\otimes\wedge(m)$ and $f_i\in \co_r\otimes\wedge(m)$ \cite[Proposition 2.12]{FK}.  This contradicts simplicity of $L$.
\end{proof}

\begin{proposition}\label{13.7.jord} If $J$ is a finite simple Jordan $H=U(\fd)$-pseudoalgebra, then as a topological Lie superalgebra, $L(\coef(J))$ is isomorphic to a current Lie superalgebra $\co_r\otimes\fs$, where $\fs$ is a simple linearly compact Lie superalgebra of growth $\dim\fd-r$.
\end{proposition}

\begin{proof}  Denote $\coef(J)$ as $\mj$.  It suffices to show that $L(\mj)$ is an irreducible central extension of the superalgebra $\co_r\otimes\fs$: since $L(\mj))$ is obtained via the TKK construction, it must be centerless.  Thus we will again follow closely the model of \cite[Theorem 13.1]{BDK}.

In order to replace $\coef(L)$ with $\ml=L(\mj)$ in the proof of Proposition~\ref{13.7.lie}, we have to demonstrate the following:

(i) $\fd\ltimes\ml$ posesses an open subalgebra containing no ideals of $\fd\ltimes\ml$ (this is an analog of \cite[Lemma 13.3]{BDK});

(ii) $\ml/Z(\ml)$ contains no proper $H$-stable ideals;

(iii) a sufficiently high power of any nonzero element of $\fd$ maps any given open subspace of $\ml$ surjectively onto $\ml$;

(iv) the action of $\fd$ on $\ml$ is transitive.

To demonstrate (i), note that just as in the proof of \cite[Lemma 13.3]{BDK}, $H(\mj_i)=\mj$ for all $i$ (here $\mj_i$ is a component in the filtration of $\mj$ induced from its construction as an annihilation algebra).  By the TKK construction this implies that for any filtration component $\ml_i$ of $\ml$, $H\ml_i=\ml$. Hence no $\ml_i$ contains an ideal of $\fd\ltimes\ml$ as such ideals are necessarily $H$-stable.

(ii) is clear because $\ml$ is in fact centerless, and any $H$-stable ideal of $\ml$ gives rise to an $H$-stable ideal of $\mj$.  By Lemma~\ref{coeffsimple} this contradicts simplicity of $J$.

For (iii), (iv), one can show as in the proof of \cite[Theorem 13.1]{BDK} that these statements hold for $\mj$.  The necessary arguments (some expressed in preceeding lemmas) carry either verbatim or with minor modifications.  Extension to $\ml$ follows from the TKK construction.
\end{proof}

We will denote the subalgebra of $\fd$ acting by derivations on $\co_r$ (and trivially on $\fs$) by $\fd'$, $H(\fd')$ by $H'$, and the dual of $H'$ by $X'$.

\begin{theorem}\label{coeff.currents}  Let $J$ be a finite simple Jordan $H=U(\fd)$-pseudoalgebra.  Then $\coef(J)$ is isomorphic to a current Jordan pseudoalgebra $\co_r\otimes\fj$, where $\fj$ is a simple linearly compact Jordan superalgebra of growth $\dim\fd-r$.
\end{theorem}

\begin{proof}  By Proposition~\ref{13.7.jord}, $\ml=L(\coef(J))$ is isomorphic to the superalgebra $\co_r\otimes\fs$. By the TKK construction, $\ml$ carries the short grading $\ml=\ml_{-1}\oplus\ml_0\oplus\ml_1$ preserved by the action of $H$. 

Denote by $\co_r'$ the augmentation ideal of $\co_r$.
The ideal $\co_r'\otimes\fs$ of $\ml$ is short-graded, thus after factoring out we obtain a short grading on $\fs$, $\fs=\fs_{-1}\oplus\fs_0\oplus\fs_1$.  Clearly $\fs_i\subset \ml_i$, i.e. $\fs_i=\fs\cap\ml_i$ (here $\ml_i$ is a component in the short grading, not a filtration component).  Moreover, by the TKK construction, $\fs_{-1}$ is a linearly compact simple Jordan superalgebra.  Denote it by $\fj$.

Let $a=x_{-1}\otimes a_{-1}+x_0\otimes a_0+x_1\otimes a_1$ lie in $\ml_{-1}$, $a_i\in\fs_i, x_i\in X'$.  Apply an element $h'$ of $H'$ such that  $h'(x_0)=1$ to $a$ and factor out $\co_r'\otimes\fs$.  It follows that $a_0\in\ml_{-1}$, i.e. $a_0=0$.  Similarly, $a_1=0$.  Thus $a=x_{-1}\otimes a_{-1}$.

Conversely, let $a\in\fs_{-1}$ and $x$ be a homogeneous elements of $X'$.  We are going to show that $a\otimes x\in \ml_{-1}$ by induction on the total degree of $x$.  Let $a_0\in\ml_0$ and $a_1\in\ml_1$ be such that $x\otimes a+a_0+a_1\ml_{-1}$.  By induction assumption, $\d a_0+\d a_1\in\ml_{-1}$ for any $\d\in\fd'$.  Thus, $\d a_i=0$ and $a_i\in\fs_i$.  Modulo $\co_r'\otimes\fs$, $a_0+a_1$ lies in $\fs_{-1}$, i.e. $a_0=a_1=0$.

Therefore, $\ml_{-1}$ is isomorphic to $\co_r\otimes\fj$ as a vector space. 

 To show that this is an isomorphism of Jordan algebras,  let $a,b\in\fj$, and let $e\in\ml_1$ be the element of the standard basis of the short subalgebra of $\ml$.  Since $e=\sum e_\alpha\otimes x_\alpha$, where $e_\alpha\in\fs_1, x_\alpha\in X'$, we have $ab=\sum [[a,e_\alpha],b]\otimes x_\alpha$.  Assume that some $x_\alpha$ are not constant, then there exists $\d\in\fd'$ that does not kill all $x_\alpha$.  Thus $\d(ab)=0$. On the other hand, $\d$ is a derivation of the Jordan algebra $\ml_{-1}$ and $\d a=\d b=0$, hence $[[a,e_\alpha],b]=0$ unless $x_\alpha\in\FF$.  It follows that the Jordan product on $\ml_{-1}$ can be restricted to $\fj$, and moreover, for $a,b\in\fj$, $(x\otimes a)(y\otimes b)=xy\otimes ab$.

To complete the proof, we remark that $\fj$ has the same growth as $\fs$.
\end{proof}

\begin{remark}  It may be possible to prove Theorem~\ref{coeff.currents} without the use of the TKK-construction.  The crucial part would be a Jordan analog of the Cartan--Guillemin theorem and, then one should proceed along the lines of \cite{BDK} as we do above.
\end{remark}

\begin{corollary}[{\cite[Proposition 5.1]{Kol}}]  Let $J$ be a simple finite Jordan non-super pseudoalgebra.  Then $L(J)$ is isomorphic to a current Lie pseudoalgebra.
\end{corollary}

\begin{proof}  According to \cite{CKj} a simple linearly compact Jordan (non-super) algebra is necessarily finite-dimensional.  Thus $\coef(J)=\co_r\otimes\fj$.  It follows that $L(\coef(J))=\co_r\otimes\fs$, hence $L(J)$ is current \cite{BDK}.  
\end{proof}

\subsection{Reconstructing conformal superalgebras}  Let $J$ be a finite Jordan conformal superalgebra.  By Corollary~\ref{confalg.reconst}, $J$ is completely determined by its annihilation algebra $\coef(J)$ together with the $\d$-action on $\coef(J)$.  Thus the classification of finite Jordan conformal superalgebras boils down to classifying their possible annihilation subalgebras with the $\d$-action.  By Theorem~\ref{coeff.currents}, $\coef(J)$ is either a simple linearly compact Jordan superalgebra of linear growth or a current algebra over a finite-dimensional Jordan superalgebra.  The classification of linearly compact Jordan superalgebras \cite{CKj} implies the following

\begin{proposition}\label{allconfcoeffs} Let $J$ be a finite Jordan conformal superalgebra.  Then $\coef(J)$ is isomorphic to one of the following superalgebras:
\begin{enumerate}
\item $J(1,n)$;
\item $\js(1,1)$;
\item $\jck(1,4)$;
\item the algebra $\FF[[t]]\otimes_{\FF}\fj$, where $\fj$ is a simple finite-dimensional Jordan superalgebra.
\end{enumerate}
\end{proposition}

It remains to describe all even surjective derivations in cases (1)--(4) (up to an automorphism).  Case (2) was essentially dealt with in Lemma~\ref{der.js}.

\begin{proposition}\label{extcurrents} Let $\d$ be an even surjective derivation of the Jordan superalgebra $\FF[[t]]\otimes\fj$, where $\fj$ is a simple finite-dimensional Jordan superalgebra.  Then by an automorphism of $\FF[[t]]\otimes\fj$, $\d$ can be conjugated to $\frac{\d}{\d t}$.
\end{proposition}

\begin{proof} For brevity denote $\FF[[t]]\otimes\fj$ by $\mj$.  It is well-known that $\d=P(t)\frac{\d}{\d t}\otimes 1+\sum t^j\otimes\d_j$, $\d_j\in\Der\fj$.  If $P(t)$ is not constant, then we get a surjective derivation on a finite-dimensional Jordan superalgebra $\mj/t\mj$ which is impossible. Thus we may assume that $P(t)=c$. Consider the automorphism $\exp(\ad \sum \frac{t^j}{j+1}\otimes\d_j)$.  Applied to $\d$, it produces a derivation of $\mj$ with a higher degree in $t$.  Since $\mj$ is a linearly compact algebra, by repeating this argument we obtain $c\frac{\d}{\d t}$ in the limit.  Then by rescaling $c=1$.
\end{proof}

The description of surjective derivations in cases (1) and (3) is similar to Proposition~\ref{extcurrents}.  We will use the embedding of the Jordan algebra into its TKK algebra and the description of even surjective derivations of $K(1,n)$ and $CK(1,4)$ in \cite{FK}.

\begin{proposition}\label{extseries}  Let $\d$ be an even surjective derivation of a Jordan superalgebra $\mj$ isomorphic to $J(1,n)$ or $\jck(1,4)$.  Then by an automorphism of $\mj$, $\d$ can be conjugated to $\frac{\d}{\d t}$.
\end{proposition}

\begin{proof}  According to Corollary~\ref{derjmn} every derivation of $J(1,n)$ has the form $\d=\delta+a\d_\theta$, where $\delta\in K(1,n)$ and $a\in P(1,n)\subset J(1,n)$.  Let $\d$ be surjective.  For every $b+c\theta\in\mj$, $(\delta+a\d_\theta)(b+c\theta)=
\delta(b)+(-1)^{|c|}ac+\delta(c)\theta$, thus we see that $P(1,n)\theta=({\mathrm Im\;}\delta)\theta$, i.e. $\delta$ is surjective.  In the case of $\mj=\jck(1,4)$, $\d=\delta$ by Corollary~\ref{derjck}.   

In both cases we will first conjugate $\delta$ to $\frac{\d}{\d t}$. 

$\mj$ is a linearly compact superalgebra, hence carries a descending filtration $\mj=\mj^{-q}\supset\ldots\mj^0\supset\mj^1\supset\ldots$.  If $\delta\mj^0\subset\mj^0$ , then $\delta\mj^1\subset\mj^1$ and $\delta$ induces a surjective derivation on a unital finite-dimensional Jordan algebra $\mj^0/\mj^1$.  This is impossible, hence $\delta\mj^0\not\subset\mj^0$

Let $\ml=L(\mj)$.  The descending filtration of $\mj$ induces the descending filtration of $\ml=\ml^{-q}\supset\ldots$.  Moreover, the extension of $\delta$ to $\ml$ is such that $\delta\ml^0\not\subset\ml^0$.  By the proof of \cite[Theorem 5.11]{FK}, $\delta=d\frac{\d}{\d t}+\ad g_0$, where $d\in\FF$ is non-zero and $g_0\in\ml^0$.   Using an inner automorphism of $\ml$, we make $\delta=d\frac{\d}{\d t}$ and then, by rescaling, we may assume $d=1$. It remains to show that such an automorphism can be choses from $\aut\mj$. 

The specific construction of automorphism in \cite[Proposition 2.13]{FK} is as follows: let $m$ be the maximal integer such that $g_0\in\ml^m\backslash\ml^{m+1}$.  Then for some $l_{m+q}\in\ml^{m+q}$, $\frac{\d}{\d t} l_{m+q}= g_0$ (here $d$ is the filtration's depth).  Applying $\exp(\ad l_{m+q})$ to $\delta$, we obtain a derivation $\frac{\d}{\d t}+$terms in $\ml^{m+q}$.  By repeating this argument, we obtain $\frac{\d}{\d t}$ in the limit.  Therefore, it suffices to show that we can always choose $l_{m+q}$ such that $\ad l_{m+q}$ is a derivation of $\mj$, i.e. that $\ad l_{m+q}|_{\sli_2}=0$.  Indeed, by the definition of $l_{m+q}$, $\frac{\d}{\d t}\ad l_{m+q}(\sli_2)=0$.  Since $l_{m+q}$ lies in $\ml_0$, $l_{m+q}=R_v+T$ where $R_v$ is the right multiplication by $v\in\mj$, and $T\in\inder\mj$.  By applying $\frac{\d}{\d t}\ad l_{m+q}$ to the identity of $\mj$, it follows that $\frac{\d}{\d t} v=0$.  Hence, $\frac{\d}{\d t}T=g_0$ and we may choose $T$ instead of $l_{m+q}$.

This completes the proof in the case of $\mj=\jck(1,4)$.  For $\mj=J(1,n)$, we obtained $\d=\frac{\d}{\d t}+a\d_\theta$ for some $a\in P(1,n)$.  Let $\int\!\!a$ be an element of $K(1,n)$ such that $\frac{\d}{\d t}(\int\!\! a)=a$.  A direct check shows that the map $\psi$ of $\mj$ that preserves $K(1,n)$ and sends $\theta$ to $\int\!\! a-\theta$ is an involutive automorphism.  To complete the proof, it suffices to check that $\psi$ conjugates $\partial$ to $\frac{\partial}{\partial t}$: indeed, $\psi$ preserves the equality $\partial|_{K(1,n)}=\frac{\partial}{\partial t}$ and $(\psi\partial\psi)(\theta)=0$.
\end{proof}

\begin{theorem}  A simple finite Jordan conformal superalgebra is isomorphic to one of the conformal superalgebras in the following list:
\begin{enumerate}
\item $J_n$;
\item $\js_1$;
\item $\jck_4$;
\item a current conformal superalgebra over a simple finite-dimensional Jordan superalgebra.
\end{enumerate}
\end{theorem}

\begin{proof} Follows from Proposition~\ref{allconfcoeffs} and, for specific cases, Lemma~\ref{der.js} and Propositions~\ref{extcurrents} and~\ref{extseries}.
\end{proof}

\end{document}